\documentclass[11pt]{article}
\usepackage{amsmath}
\usepackage{amsthm}
\usepackage{amssymb,amsfonts}
\usepackage{hyperref}
\usepackage{latexsym}
\usepackage{todonotes}
\usepackage{nicefrac}
\usepackage{multicol}
\usepackage{epsfig}
\usepackage{stmaryrd}
\usepackage{setspace}
\usepackage{enumerate}
\usepackage[all]{xypic}
\usepackage{bbm,ifpdf,tikz}
\usetikzlibrary{calc,decorations.pathmorphing,shapes}
\ifpdf
\usepackage{pdfsync}
\fi
\usepackage{listings}
\newtheorem{theorem}{Theorem}[section]

\newtheorem{lemma}[theorem]{Lemma}
\newtheorem{proposition}[theorem]{Proposition}
\newtheorem{corollary}[theorem]{Corollary}

{
\theoremstyle{definition}

\newtheorem{example}[theorem]{Example}

\newtheorem{remark}[theorem]{Remark}
}

\newcommand{\excise}[1]{}

\renewcommand{\and}{\qquad\text{and}\qquad}

\newcommand{\Q}{\mathbb{Q}}

\newcommand{\bracenom}{\genfrac{\lbrace}{\rbrace}{0pt}{}}

\newcommand{\sst}[1]{{\left[\!\left[ \begin{matrix} #1 \end{matrix} \right]\!\right]}}
\newcommand{\beq}{\begin{eqnarray*}}
\newcommand{\eeq}{\end{eqnarray*}}

\begin{document}
\spacing{1.2}
\noindent{\LARGE\bf On the enumeration of series-parallel matroids
}\\

\noindent{\bf Nicholas Proudfoot\footnote{Supported by NSF grants DMS-1954050, DMS-2053243, and DMS-2344861.}, Yuan Xu\footnote{Supported by Simons Foundation Collaboration Grant \#849676.}, and Benjamin Young}\\
Department of Mathematics, University of Oregon,
Eugene, OR 97403\\

{\small
\begin{quote}
\noindent {\em Abstract.} 
By the work of Ferroni and Larson, Kazhdan--Lusztig polynomials and $Z$-polynomials of complete graphs have combinatorial interpretations in terms of quasi series-parallel matroids.
We provide explicit formulas for the number of series-parallel matroids and the number of simple series-parallel matroids of a given rank and cardinality,
extending results of Ferroni--Larson and Gao--Proudfoot--Yang--Zhang.
\end{quote} }

\section{Introduction}
Given a graph, a {\bf series extension} is a graph obtained by subdividing an edge, and a {\bf parallel extension} is a graph obtained by adding a new edge parallel to an existing one.
A graph is called {\bf series-parallel} if it can be constructed from a 2-cycle by a sequence of series and parallel extensions.  By convention, a single edge
and a single loop
are also considered series-parallel graphs.  A matroid associated with a series-parallel graph is called a {\bf series-parallel matroid}.  A series-parallel matroid is {\bf simple} if and only if it comes from
a graph with no loops or parallel edges.

A (possibly empty) direct sum of series-parallel matroids is called {\bf quasi series-parallel}; this is the same as taking matroids associated with disjoint unions of series-parallel graphs.  A quasi series-parallel matroid is
simple if and only if each of its components is simple.  Quasi series-parallel matroids are characterized by the property of having no minors equal to the uniform matroid of rank 2 on 4 elements or
the matroid associated with the complete graph $K_4$ \cite[Proposition 2.1]{FL-braid}.  The {\bf rank} of a quasi series-parallel matroid is equal to the number of vertices minus the number of connected components of the corresponding graph.

Consider the following quantities:
\begin{align*}
C_{n,k} &= \text{the number of series-parallel matroids on $[n]$ of rank $k$ }& \text{\cite[A140945]{oeis}}\\
E_{n,k} &= \text{the number of simple series-parallel matroids on $[n]$ of rank $k$} & \text{\cite[A361355]{oeis}}\\
A_{n,k} &= \text{the number of quasi series-parallel matroids on $[n]$ of rank $k$} & \text{\cite[A359985]{oeis}}\\
S_{n,k} &= \text{the number of simple quasi series-parallel matroids on $[n]$ of rank $k$} & \text{\cite[A361353]{oeis}}
\end{align*}

\begin{remark}
The letter $A$ stands for {\bf All} quasi series-parallel matroids, $S$ stands for {\bf Simple} quasi series-parallel matroids, and $C$ stands for {\bf Connected} quasi series-parallel matroids,
which are the same as series-parallel matroids (with the convention that the empty matroid is not connected).  The letter $E$ does not stand for anything, but it means simple and connected.
In \cite{FL-braid}, the quantity $E_{2k,k+1}$ is denoted $E_k$.
\end{remark}

\begin{remark}
The original motivation for studying these quantities is that 
$A_{n,k}$ (respectively $S_{n,k}$) is equal to the coefficient of $t^{n-k}$ in the $Z$-polynomial (respectively Kazhdan--Lusztig polynomial)
of the matroid associated with the complete graph $K_{n+1}$ \cite[Theorem 1.1]{FL-braid}.  This is the only known combinatorial description of these coefficients.
\end{remark}

\begin{remark}
Note that the number of series-parallel matroids on $[n]$ is not the same as the number of series-parallel graphs with edge set $[n]$, because different graphs can induce the same matroid.
For example, there are three different ways (up to isomorphism) to label the edges of the 4-cycle with the labels $\{1,2,3,4\}$, but they all induce the uniform matroid of rank 3.
\end{remark}

Consider the following generating functions:
\beq E(x,y) &:=& \sum_{n=1}^\infty \sum_{k=0}^n E_{n,k}\, y^k \frac{x^n}{n!},\qquad S(x,y) := \sum_{n=0}^\infty \sum_{k=0}^n S_{n,k}\, y^k \frac{x^n}{n!}\\
C(x,y) &:=& \sum_{n=1}^\infty \sum_{k=0}^n C_{n,k}\, y^k \frac{x^n}{n!},\qquad A(x,y) := \sum_{n=0}^\infty \sum_{k=0}^n A_{n,k}\, y^k \frac{x^n}{n!}.
\eeq
Note that the two generating functions on the left begin with $n=1$, while the two on the right begin with $n=0$;
this is because the empty matroid is quasi series-parallel but not series-parallel.
The combinatorial relationships between these numbers can be expressed in terms of their generating functions.

\begin{proposition}\label{gf}
We have the following identities:\footnote{With the fourth equality, we fix a sign error from an equation appearing in the proof of \cite[Proposition 2.14]{FL-braid}.}
\vspace{-2\baselineskip}
\begin{multicols}{2}

\beq S(x,y) &=& e^{E(x,y)}\\
A(x,y) &=& e^{C(x,y)}\\
C(x,y) &=& E(e^x - 1,y) +  x\\
A(x,y) &=& S(e^x - 1, y)\cdot e^x
\eeq

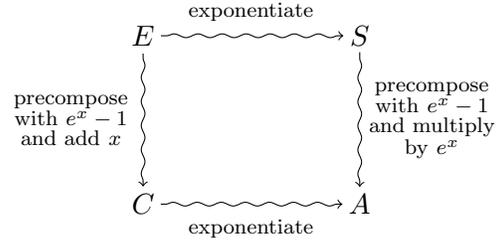
\columnbreak

\begin{center}
\begin{tikzpicture}[scale = .8]
	\node[circle, inner sep=.02cm](A)at(1.8,-1.4){$A$};
	\node[circle, inner sep=.02cm](C)at(-1.8,-1.4){$C$};
	\node[circle, inner sep=.02cm](S)at(1.8,1.4){$S$};
	\node[circle, inner sep=.02cm](E)at(-1.8,1.4){$E$};
	\node[circle, inner sep=.02cm](post1)at(0,1.8){{\text{\em\scriptsize exponentiate}}};
	\node[circle, inner sep=.02cm](post2)at(0,-1.8){{\text{\em\scriptsize exponentiate}}};
	\node[circle, inner sep=.02cm](pre1)at(3,0){$\substack{\text{{\em precompose}}\\ \text{{\em with} $e^x-1$}\\ \text{{\em and multiply}}\\ \text{{\em by} $e^x$}}$};
	\node[circle, inner sep=.02cm](pre2)at(-3,0){$\substack{\text{{\em precompose}}\\ \text{{\em with} $e^x-1$}\\ \text{{\em and add} $x$}}$};
	\draw[<-,decorate,
  decoration={snake,amplitude=.7pt,segment length=3.2mm,pre=lineto,pre length=2pt}](A)to(S);
	\draw[<-,decorate,
  decoration={snake,amplitude=.7pt,segment length=3.2mm,pre=lineto,pre length=2pt}](A)to(C);
	\draw[<-,decorate,
  decoration={snake,amplitude=.7pt,segment length=3.2mm,pre=lineto,pre length=2pt}](S)to(E);
	\draw[<-,decorate,
  decoration={snake,amplitude=.7pt,segment length=3.2mm,pre=lineto,pre length=2pt}](C)to(E);
\end{tikzpicture}
\end{center}

\end{multicols}
\end{proposition} 

\vspace{-2\baselineskip}
\begin{proof}
A quasi series-parallel matroid on $[n]$ is given by a partition of $[n]$ along with a series-parallel matroid on each part, and
it is simple if and only if each component is simple.  This fact, combined with \cite[Corollary 5.1.6]{EC2}, implies the first two identities.
When $n\geq 2$, a series-parallel matroid on $[n]$ is given by a partition of $[n]$ into parallel classes and a simple series-parallel matroid on the set of parallel classes.
This observation, combined with \cite[Theorem 5.1.4]{EC2}, implies the third identity.
(The addition of $x$ comes from the matroid of rank 0 on the set $[1]$, which is series-parallel but not simple.)
Finally, a quasi series-parallel matroid on $[n]$ is given by a set of loops, a partition of the nonloops into parallel classes, and a simple series-parallel matroid on the set of parallel classes.
This statement implies the fourth identity by \cite[Proposition 5.1.1 and Theorem 5.1.4]{EC2}, with the factor of $e^x$ corresponding to the choice of the set of loops.
\end{proof}

We focus here on the numbers $E_{n,k}$, from which all of the others can be computed.
We know that we have $E_{n,k} = 0$ when $n \geq 2k > 0$ \cite[Proposition 2.10]{FL-braid}.
Theorem \ref{existing} provides formulas for $E_{2k-1,k}$ \cite[Corollary 2.12]{FL-braid} and $E_{2k-2,k}$
\cite[Corollary 1.6]{GPYZ}.  We adopt the standard notation
$(2k-1)!! := 1\cdot 3 \cdot 5 \cdots (2k-1) = \frac{(2k)!}{2^k k!}$.

\begin{theorem}\label{existing}{\em \cite{FL-braid,GPYZ}}
We have $$\frac{E_{2k-1,k}}{(2k-1)!!} = (2k-1)^{k-3}\and  \frac{E_{2k-2,k}}{(2k-3)!!}
= (2k-1)^{k-2} - (2k-2)^{k-2} + \frac 2 3 (k-2)(2k-2)^{k-3}.$$
\excise{
\begin{itemize}
\item $\displaystyle \frac{E_{2k-1,k}}{(2k-1)!!} = (2k-1)^{k-3}$
\item $\displaystyle \frac{E_{2k-2,k}}{(2k-3)!!}
= (2k-1)^{k-2} - (2k-2)^{k-2} + \frac 2 3 (k-2)(2k-2)^{k-3}.$
\end{itemize}
}
\end{theorem}

Our goal in this note is to provide a formula for $E_{2k-r,k}$ for arbitrary $k$ and $r$.
Our formula becomes more complicated as $r$ grows.  It can be used to recover Theorem \ref{existing},
and we also use it to provide an explicit closed formula for the next case $E_{2k-3,k}$ (Example \ref{next}).

Consider
the {\bf unsigned associated Stirling number of the first kind}
\begin{equation}\label{recursion}\sst{n\\k} = (n-1)\sst{n-2\\k-1} + (n-1) \sst{n-1\\k},\end{equation}
which counts the number of derangements
of $[n]$ with $k$ cycles \cite[page 256]{Comtet}.  This quantity vanishes when $n< 2k$, and Equation \eqref{recursion}
implies the following formulas when $n$ is close to $2k$:
$$\sst{2k\\k} = (2k-1)!!,\quad \sst{2k+1\\k} = \frac 2 3 k\, (2k+1)!!,\quad\text{and}\quad\sst{2k+2\\k} = \frac  1 9 (4k + 5)(k+1)k\, (2k+1)!!.$$

\begin{theorem}\label{nick E formula}
For all $0\leq r\leq k$, we have $$E_{2k-r,k} = \sum_{p=1}^r \sst{2k-p-1\\k-p} \sum_{i=0}^{r-p} \frac{(-1)^{i+p+1}(2k-p-i)^{k-p-1}}{i! (r-p-i)!}.$$
\end{theorem}

\begin{example}\label{next}
When $r=1$ and $r=2$, Theorem \ref{nick E formula} reproduces Theorem \ref{existing}.
When $r=3$, Theorem \ref{nick E formula} tells us that
\begin{eqnarray*}\frac{E_{2k-3,k}}{(2k-3)!!} 
&=&  \frac 1 2 (2k-1)^{k-2} - (2k-2)^{k-2} + \frac 1 2 (2k-3)^{k-2}\\
&& + \frac 2 3 (k-2) \left((2k-3)^{k-3} - (2k-2)^{k-3}\right)\\
&& + \frac  1 9 (4k - 7)(k-2)(k-3) (2k-3)^{k-5}.
\end{eqnarray*}
\end{example}

\begin{remark}
Let $M$ be a simple quasi series-parallel matroid of rank $k$ on the set $[2k-r]$, and let $\{M_i\}$ be its connected components.
Then $M_i$ is a simple series-parallel matroid of rank $k_i$ on a set of cardinality $2k_i - r_i$,
and we have $\sum_i k_i = k$ and $\sum_i r_i = r$.  Thus $S_{2k-r,k}$ may be computed in terms of $E_{2j-s,j}$ for $j\leq k$ and $s\leq r$.
The precise formula can be derived from the first equation in Proposition \ref{gf}.
\end{remark}

We prove Theorem \ref{nick E formula} using the generating functions.  Ferroni and Larson provide an expression for the generating function $C(x,y)$ in terms of the compositional inverse of the function 
$$\frac{1}{y} \log(1+xy) + \log(1+x) - x,$$ where $y$ is regarded as a parameter (Section \ref{our case}).  We explicitly compute the coefficients of this compositional inverse, which gives
us a formula for the numbers $C_{n,k}$ (Corollary \ref{C}).  
We then combine this with the third identity in Proposition \ref{gf} to prove Theorem \ref{nick E formula}.

\vspace{\baselineskip}
\noindent
{\em Acknowledgments:}
The authors are grateful to Luis Ferroni and Matt Larson, whose work made this paper possible.

\section{Two Stirling lemmas}
We begin with two lemmas about Stirling numbers that we will need later in the paper.  Let $\bracenom{n}{k}$ be the {\bf Stirling number of the second kind},
which counts partitions of $[n]$ into $k$ nonempty parts.

\begin{lemma}\label{Yuan2.3}
We have $$\sum_{p=0}^\ell (-1)^{\ell + p} \binom{m+p}{\ell+p} \sst{\ell+p\\p} = \bracenom{m+1}{m-\ell+1}.$$
\end{lemma}

\begin{proof}
Let us denote the left-hand side of the equation by $T_{m,\ell}$.
We have
$$\bracenom{m+1}{m-\ell+1} - \bracenom{m}{m-\ell} = (m-\ell+1)\bracenom{m}{m-\ell+1},$$
and we will show that
$T_{m,\ell}$ satisfies the same recursion.  Indeed, we have 
\beq
 T_{m,\ell} - T_{m-1,\ell}  &=& \sum_{p=1}^\ell (-1)^{p+\ell} \left(\binom{m+p}{\ell+p} - \binom{m-1+p}{\ell+p}\right) \sst{\ell+p \\p} \\
&=& \sum_{p=0}^\ell (-1)^{p+\ell} \binom{m+p-1}{\ell + p-1} \sst{\ell+p \\p} \\
   & =&  \sum_{p=0}^\ell (-1)^{p+\ell} \binom{m+p-1}{\ell + p-1} (\ell+p-1) \left(\sst{\ell+p-2\\p-1} +\sst{\ell+p-1 \\p} \right)\\
   & =& (m-\ell+1) \sum_{p=0}^\ell (-1)^{p+\ell} \binom{m+p-1}{\ell + p-2}\left(\sst{\ell+p-2\\p-1} +\sst{\ell+p-1 \\p} \right)\\
&=& (m-\ell+1) \sum_{q=0}^{\ell-1} (-1)^{q+\ell} \left(\binom{m+q-1}{\ell+q-2} - \binom{m+q}{\ell+q-1} \right)\sst{\ell+q-1\\q}\\
&=& (m-\ell+1) \sum_{q=0}^{\ell-1} (-1)^{q+\ell-1} \binom{m+q-1}{\ell+q-1}\sst{\ell+q-1\\q}\\
&=& (m-\ell+1) T_{m-1,\ell-1}.
\eeq   
This completes the proof.
\end{proof}

\begin{lemma}\label{Yuan2.6}
We have $$\bracenom{n+k}{m} 
= \sum_{j=0}^{k-1}\bracenom{n+1}{m-j} \sum_{i=0}^j  \frac{(-1)^i(m-i)^{k-1}}{i!(j-i)!}.$$
\end{lemma}

\begin{proof}
We have
\begin{eqnarray*}
m!\bracenom{n+k}{m} &=& \big{|}\{f:[n+k]\twoheadrightarrow [m]\}\big{|}\\
&=&\sum_{j=1}^{k-1} \binom m j \big{|}\{f:[n+1]\twoheadrightarrow [m-j]\}\big{|}\cdot \big{|}\{f:[k-1]\to[m] \mid [j]\subset\operatorname{im}(f)\}\big{|}\\
&=&\sum_{j=1}^{k-1} \binom m j (m-j)! \bracenom{n+1}{m-j}\cdot \big{|}\{f:[k-1]\to[m] \mid [j]\subset\operatorname{im}(f)\}\big{|},
\end{eqnarray*} 
and therefore $$\bracenom{n+k}{m} = \sum_{j=1}^{k-1} \frac{1}{j!} \bracenom{n+1}{m-j}\cdot \big{|}\{f:[k-1]\to[m] \mid [j]\subset\operatorname{im}(f)\}\big{|}.$$
By the inclusion-exclusion principle,
\begin{eqnarray*}\big{|}\{f:[k-1]\to[m] \mid [j]\subset\operatorname{im}(f)\}\big{|} &=& \sum_{i=0}^{k-1}(-1)^i \binom j i \big{|}\{f:[k-1]\to[m] \mid [i]\not\subset\operatorname{im}(f)\}\big{|}\\
&=& \sum_{i=0}^{k-1}(-1)^i \binom j i \big{|}\{f:[k-1]\to[m-i]\}\big{|}\\
&=& \sum_{i=0}^{k-1}(-1)^i \binom j i (m-i)^{k-1}.
\end{eqnarray*}
This completes the proof.
\end{proof}

\section{Sums of products of reciprocals}
Consider the numbers $$H_{m,k} := \sum_{\substack{j_1 + \cdots + j_k = m\\ j_1\geq 1,\ldots,j_k\geq 1}} \,\frac{1}{(j_1 + 1)\cdots (j_k+1)}.$$

\begin{lemma}\label{new recursion}
We have the recursion
$$n H_{n-k,k} = k H_{n-k-1,k-1} + (n-1) H_{n-k-1,k}.$$
\end{lemma}

\begin{proof}
We have
\beq
n H_{n-k,k} &=& \frac{n!}{k!} \sum_{\substack{j_1 + \cdots + j_k = n-k\\ j_1\geq 1,\ldots,j_k\geq 1}} \,\frac{1}{(j_1 + 1)\cdots (j_k+1)}\\
&=& \sum_{\substack{j_1 + \cdots + j_k = n-k\\ j_1\geq 1,\ldots,j_k\geq 1}} \,\frac{(j_1+1)+\cdots+(j_k+1)}{(j_1 + 1)\cdots (j_k+1)}.\\
\eeq
By symmetry, we may replace the numerator in the fraction above by $k (j_{k}+1)$, and we obtain the equation
\beq
n H_{n-k,k} 
&=& \sum_{\substack{j_1 + \cdots + j_k = n-k\\ j_1\geq 1,\ldots,j_k\geq 1}} \,\frac{k(j_k+1)}{(j_1 + 1)\cdots (j_k+1)}\\
&=& \sum_{\substack{j_1 + \cdots + j_k = n-k\\ j_1\geq 1,\ldots,j_k\geq 1}} \,\frac{k}{(j_1 + 1)\cdots (j_{k-1}+1)}\\
&=& \sum_{j_k\geq 1}\;\; \sum_{\substack{j_1 + \cdots + j_{k-1} = n-k-j_k\\ j_1\geq 1,\ldots,j_{k-1}\geq 1}} \,\frac{k}{(j_1 + 1)\cdots (j_{k-1}+1)}.
\eeq
Similarly, we have
$$(n-1) H_{n-k-1,k} =  \sum_{j_k\geq 1}\;\; \sum_{\substack{j_1 + \cdots + j_{k-1} = n-k-j_k-1\\ j_1\geq 1,\ldots,j_{k-1}\geq 1}} \,\frac{k}{(j_1 + 1)\cdots (j_{k-1}+1)}.$$
Taking the difference, we find that
\beq n H_{n-k,k}  - (n-1) H_{n-k-1,k} &=& \sum_{\substack{j_1 + \cdots + j_{k-1} = n-k-1\\ j_1\geq 1,\ldots,j_{k-1}\geq 1}} \,\frac{k}{(j_1 + 1)\cdots (j_{k-1}+1)}\\
&=& k H_{n-k-1,k}.
\eeq
This completes the proof.
\end{proof}

\begin{lemma}\label{Yuan2.1}
We have $$H_{n-k,k} = \frac{k!}{n!} \sst{n\\ k}.$$
\end{lemma}

\begin{proof}
The recursion in Equation \eqref{recursion} matches the one in Lemma \ref{new recursion}.
\end{proof}

\begin{remark}
There is a direct proof of Lemma \ref{Yuan2.1}, not requiring Lemma \ref{new recursion},
making use instead of a comment by Copeland in \cite[A008306]{oeis}.  We thank the referee for this observation.
\end{remark}

\begin{lemma}\label{Yuan2.2}
We have
$$\sum_{\substack{j_1 + \cdots + j_k = m\\ j_1\geq 1,\ldots,j_k\geq 1}}\, \prod_{i=1}^k\frac{1+y^{j_i}}{j_i + 1} = 
\sum_{\ell = 0}^{m} y^\ell\sum_{p=0}^k \binom k p  H_{\ell,p} H_{m-\ell,k-p}.$$
\end{lemma}

\begin{proof}
We have
\beq 
\sum_{\substack{j_1 + \cdots + j_k = m\\ j_1\geq 1,\ldots,j_k\geq 1}} \,\prod_{i=1}^k\frac{1+y^{j_i}}{j_i + 1} 
&=& \sum_{\substack{j_1 + \cdots + j_k = m\\ j_1\geq 1,\ldots,j_k\geq 1}}\,\sum_{p=0}^k \binom k p \frac{y^{j_1+\cdots + j_p}}{(j_1+1)\cdots (j_k+1)}\\
&=& \sum_{p=0}^k \binom k p \sum_{\ell = p}^{m-k+p} y^\ell \sum_{\substack{j_1 + \cdots + j_k = m\\  j_1 + \cdots + j_p = \ell\\ j_1\geq 1,\ldots,j_k\geq 1}}\, \frac{1}{(j_1+1)\cdots (j_k+1)}\\
&=& \sum_{\ell = 0}^{m} y^\ell\sum_{p=0}^k \binom k p  H_{\ell,p} H_{m-\ell,k-p}.
\eeq
This completes the proof.
\end{proof}

Combining Lemmas \ref{Yuan2.1} and \ref{Yuan2.2} yields the following corollary, which we will use in Section \ref{our case}.

\begin{corollary}\label{H}
We have
$$\sum_{\substack{j_1 + \cdots + j_k = m\\ j_1\geq 1,\ldots,j_k\geq 1}}\, \prod_{i=1}^k\frac{1+y^{j_i}}{j_i + 1} = 
\frac{k!}{(m-k)!}\sum_{\ell=0}^m y^\ell  \sum_{p=0}^{k}\binom{m-k}{\ell+p} \sst{\ell+p\\ p} \sst{m-\ell+k-p\\ k-p}.$$

\end{corollary}

\section{Inverting a power series}\label{our case}
The {\bf partial Bell polynomials} $B_{n,k}(t_1,\ldots,t_{n-k+1})$ are characterized by the identity
\begin{equation}\label{Bell def}\exp\left(y\sum_{j=1}^\infty t_j \frac{x^j}{j!}\right) = \sum_{0\leq k\leq n} B_{n,k}(t_1,\ldots,t_{n-k+1})\,y^k \frac{x^n}{n!}.\end{equation}
The following lemma gives an explicit expression for these polynomials.

\begin{lemma}\label{YuanBell}
We have
$$B_{n,k}(t_1,\ldots,t_{n-k+1}) = \frac{n!}{k!}\sum_{\substack{j_1 + \cdots + j_k = n\\ j_1\geq 1,\ldots,j_k\geq 1}} \frac{t_{j_1}}{j_1!}\cdots \frac{t_{j_k}}{j_k!}.$$
\end{lemma}

\begin{proof}
Equation \eqref{Bell def} implies that $B_{n,k}(t_1,\ldots,t_{n-k+1})$ is equal to the coefficient of $x^n$ in the power series
$$\frac{n!}{k!} \left(\sum_{j=1}^\infty t_j \frac{x^j}{j!}\right)^k.$$
The lemma follows.
\end{proof}

Suppose that $$F(x) = \sum_{n=1}^\infty F_n \frac{x^n}{n!}
\and G(x) = \sum_{n=1}^\infty G_n \frac{x^n}{n!}$$
are power series with coefficients in some commutative $\Q$-algebra $R$.
Suppose further that $F_1 \neq 0$, and let $\hat{F_n} = \frac{F_{n+1}}{(n+1)F_1}$,
so that $$\hat{F}(x) := \sum_{n=1}^\infty \hat{F_n} \frac{x^n}{n!} = \frac{F(x) - F_1x}{x}.$$

The following result is a corollary of the Lagrange inversion theorem \cite[Corollary 11.3]{Charalambides}.

\begin{theorem}\label{invert}
We have $G(F(x)) = x$ if and only if $G_1 = F_1^{-1}$ and, for all $n>1$, 
\beq G_n &=& \frac{1}{F_1^n}\sum_{k=1}^{n-1}n(n+1)\cdots(n+k-1)B_{n-1,k}\left(\hat F_1,\ldots, \hat F_{n-k}\right)\\
&=& \frac{1}{F_1^n} \sum_{k=1}^{n-1} (-1)^k \frac{(n+k-1)!}{k!} \sum_{\substack{j_1 + \cdots + j_k = n-1\\ j_1\geq 1,\ldots,j_k\geq 1}} \prod_{i=1}^k \frac{\hat F_{j_i}}{j_i!}.\eeq
\end{theorem}

We now apply Theorem \ref{invert} to a particular power series with coefficients in the commutative $\Q$-algebra $\Q[y]$.
Let $$F(x,y) = \sum_{n=1}^\infty F_n(y) \frac{x^n}{n!} := \frac{1}{y} \log(1+xy) + \log(1+x) - x.$$
Explicitly, we have $F_1(y)=1$ and $F_n(y) = (-1)^{n-1}(n-1)!(1+y^{n-1})$ for all $n>1$.
Let
$$G(x,y) = \sum_{n=1}^\infty G_n(y) \frac{x^n}{n!} = 
\sum_{n=1}^\infty \sum_{k=0}^{\infty} G_{n,k}\, y^k\frac{x^n}{n!}$$
be the unique power series with the property that $G(F(x,y),y) = x$.

\begin{proposition}\label{hard part}
We have
$$G_{n,\ell} 
= G_{n,n-\ell-1}  = \sum_{j=0}^{\ell} (-1)^{j+\ell} \sst{j+\ell\\j} \bracenom{n+j }{j+\ell+1}.$$
\end{proposition}

\begin{proof}
Let $$\hat F_n(y) := \frac{F_{n+1}(y)}{(n+1)F_1(y)} = \frac{(-1)^n n! (1+y^n)}{n+1}.$$
By Theorem \ref{invert}, we have
\beq G_n(y) &=& \sum_{k=1}^{n-1} (-1)^k \frac{(n+k-1)!}{k!} \sum_{\substack{j_1 + \cdots + j_k = n-1\\ j_1\geq 1,\ldots,j_k\geq 1}} \prod_{i=1}^k \frac{\hat F_{j_i}(y)}{j_i!}\\
&=&  \sum_{k=1}^{n-1} (-1)^{n+k-1} \frac{(n+k-1)!}{k!} \sum_{\substack{j_1 + \cdots + j_k = n-1\\ j_1\geq 1,\ldots,j_k\geq 1}} \prod_{i=1}^k\frac{1+y^{j_i}}{j_i + 1}.
\eeq
Note that this polynomial is clearly palindromic of degree $n-1$, which implies that $G_{n,\ell} = G_{n,n-\ell-1}$.
By Corollary \ref{H}, $G_n(y)$ is equal to
$$\sum_{\ell = 0}^{n-1} y^\ell\ \sum_{k=1}^{n-1} (-1)^{n+k-1} \sum_{p=0}^{\ell} \sst{\ell + p\\ p}\sst{n-1-\ell+k-p\\ k-p} \binom{n+k-1}{\ell+p}.
$$
Taking the coefficient of $y^\ell$ and reindexing with $j=k-p$, we get
$$G_{n,\ell} = \sum_{j=0}^{n-\ell-1} (-1)^{n+j-\ell-1} \sst{n-1-\ell + j\\ j}\sum_{p=0}^\ell (-1)^{\ell+p} \sst{\ell + p\\ p}\binom{n-1+j+p}{\ell + p}.$$
Note that the symmetry 
$G_{n,\ell} = G_{n, n-1-\ell}$ can be seen by exchanging $j$ and $p$ in the summation above.
By Lemma \ref{Yuan2.3} with $m=n-1+j$, we have
$$G_{n,\ell} = \sum_{j=0}^{n-\ell-1} (-1)^{n+j-\ell-1} \sst{n-1-\ell + j\\ j}\bracenom{n+j}{n+j-\ell}.$$
Replacing $\ell$ with $n-1-\ell$ allows us to rewrite our expression as
$$G_{n,\ell} = G_{n,n-1-\ell} = \sum_{j=0}^{\ell} (-1)^{j+\ell} \sst{j+\ell\\j} \bracenom{n+j }{j+\ell+1}.$$
This completes the proof.
\end{proof}

Proposition \ref{hard part}, along with a theorem of Ferroni and Larson, provides a formula for $C_{n,\ell}$.

\begin{corollary}\label{C}
For all $n \geq 2$, we have
$$C_{n,\ell} = \sum_{k=0}^{\ell-1} (-1)^{k+\ell-1} \sst{k+\ell-1\\k} \bracenom{n-1+k }{k+\ell}.$$
\end{corollary}

\begin{proof}
Using the work of Drake \cite[Example 1.5.1]{Drake},
Ferroni and Larson \cite[Proposition 2.3]{FL-braid} show that
$$ C(x,y) =  (1+y)x + y \int G(x,y)\, dx,$$
where the improper integral is taken to have no constant term. This means that,
for all $n\geq 2$, 
$C_{n,\ell} = G_{n-1,\ell-1}$.  The Corollary then follows from Proposition \ref{hard part}.
\end{proof}

\begin{remark}
In Proposition \ref{hard part}, we gave an algebraic proof of the identity $G_{n,\ell} = G_{n,n-1-\ell}$.
We can reinterpret this identity as saying that $C_{n+1,\ell+1} = C_{n+1,n-\ell}$, which follows from the fact that matroid duality
is a bijection from the set of series-parallel matroids on $[n+1]$ of rank $\ell+1$ to the set of series-parallel matroids on $[n+1]$ of rank $n-\ell$.
\end{remark}

\section{Proof of Theorem \ref{nick E formula}}
This section is devoted to using Corollary \ref{C} to prove Theorem \ref{nick E formula}.

\begin{lemma}\label{C to E}
For all $n \geq 2$, we have
$$C_{n,\ell} = \sum_{m=\ell}^n \bracenom n m E_{m,\ell}.$$
\end{lemma}

\begin{proof}
This can be derived from the third identity in Proposition \ref{gf}, or one can prove it directly using the same combinatorial reasoning employed in the proof of Proposition \ref{gf}.
That is, a series-parallel matroid on $[n]$ is given by a partition of $[n]$ into $m$ parallel classes for some $m$, along with a simple series-parallel matroid on the set of parallel classes.  The lemma follows.
\end{proof}

Let $$\tilde E_{n,\ell} := \sum_{p=1}^{2\ell-n} \sst{2\ell-p-1\\ \ell-p} \sum_{i=0}^{2\ell-n-p} \frac{(-1)^{i+p+1}(2\ell-p-i)^{\ell-p-1}}{i! (2\ell-n-p-i)!},$$
so that $$\tilde E_{2k-r,k} = \sum_{p=1}^r \sst{2k-p-1\\k-p} \sum_{i=0}^{r-p} \frac{(-1)^{i+p+1}(2k-p-i)^{k-p-1}}{i! (r-p-i)!}$$
is the expression appearing on the right-hand side of the equation in the statement of the theorem.
We next prove the analogue of Lemma \ref{C to E} for $\tilde E$.

\begin{lemma}\label{tilde recurse}
For all $n \geq 2$, we have
$$C_{n,\ell} = \sum_{m=\ell}^n \bracenom n m \tilde E_{m,\ell}.$$
\end{lemma}

\begin{proof}
By Corollary \ref{C} and using Lemma \ref{Yuan2.6}, we have
$$C_{n,\ell} = 
\sum_{k=0}^{\ell-1} (-1)^{k+\ell-1} \sst{k+\ell-1\\k} \sum_{j=0}^{k-1}\bracenom{n}{k+\ell-j} \sum_{i=0}^j \frac{(-1)^i(k+\ell-i)^{k-1}}{i!(j-i)!},$$
Setting $m = k + \ell - j$, we get  
$$C_{n,\ell} = \sum_{m=1}^n \bracenom{n}{m} \sum_{k=0}^{\ell-1} (-1)^{k+\ell-1} \sst{k+\ell-1\\k}\sum_{i=0}^{k+\ell-m} (-1)^i \frac{(k+\ell-i)^{k-1}}{i!(k+\ell-m-i)!},$$
thus it will suffice to show that
$$\sum_{k=0}^{\ell-1} (-1)^{k+\ell-1} \sst{k+\ell-1\\k}\sum_{i=0}^{k+\ell-m} \frac{(-1)^i(k+\ell-i)^{k-1}}{i!(k+\ell-m-i)!}$$
is equal to 
$$\sum_{p=1}^{2\ell-m} \sst{2\ell-p-1\\ \ell-p} \sum_{i=0}^{2\ell-m-p} \frac{(-1)^{i+p+1}(2\ell-p-i)^{\ell-p-1}}{i! (2\ell-m-p-i)!}.$$
This is readily seen by setting $k=\ell-p$.
\end{proof}

\begin{proof}[Proof of Theorem \ref{nick E formula}]
We need to prove that $E_{n,\ell} = \tilde E_{n,\ell}$ for all $n\geq \ell\geq 1$.
We fix $\ell\geq 1$ and proceed by induction on $n$.  If $n=\ell=1$, we can verify the equality directly.  Otherwise we have $n\geq 2$, so Equation \eqref{C to E} and Lemma \ref{tilde recurse} tell us that
$$\bracenom n \ell E_{\ell,\ell} + \bracenom{n}{\ell+1} E_{\ell+1,\ell} + \cdots + \bracenom{n}{n} E_{n,\ell} = C_{n,\ell}
= \bracenom n \ell \tilde E_{\ell,\ell} + \bracenom{n}{\ell+1} \tilde E_{\ell+1,\ell} + \cdots + \bracenom{n}{n} \tilde E_{n,\ell}.$$
By our inductive hypothesis, we can conclude that $E_{n,\ell} = \tilde E_{n,\ell}$.
\end{proof}

\bibliography{./symplectic}
\bibliographystyle{amsalpha}

\end{document}